\documentclass{amsart}
\usepackage{amsthm,amsmath,amssymb, color, graphics}
\usepackage{tikz-cd}
\usepackage{mathrsfs}
\usepackage{adjustbox}
\usepackage{hyperref}
 \usepackage{amsaddr}

\newtheorem{thm}{Theorem}[section]
\newtheorem{claim}[thm]{Claim}
\newtheorem{cor}[thm]{Corollary}

\newtheorem{lem}[thm]{Lemma}
\newtheorem{prop}[thm]{Proposition}

\newtheorem{question}[thm]{Question}
\theoremstyle{definition}
\newtheorem{defn}[thm]{Definition}

\newcommand\G{\mathbf{G}}

\newcommand{\Om}{\Omega}

\begin{document}

\title{All you need is $\mathbf{A}_\kappa$}
\author{Nathaniel Bannister}
\address{Carnegie Mellon University}
\begin{abstract}
    We show that the vanishing of higher derived limits of the system $\mathbf{A}_\kappa$ implies the additivity of strong homology on the class of locally compact metric spaces of weight at most $\kappa$, thereby establishing a converse to a theorem of Marde\v{s}i\'c and Prasolov. 
\end{abstract}

\thanks{
The author would like to thank James Cummings for helpful discussions leading to this paper and the anonymous referee for a prompt and thorough report. 
}
\maketitle

\section{Introduction}

Strong homology is a theory originally defined by Marde\v{s}i\'c, and offers a correction to \v{C}ech homology which recovers the exactness axiom. 
While investigating strong homology, Marde\v{s}i\'c and Prasolov defined in \cite{SHINA} the inverse system $\mathbf{A}$ and showed that the additivity of strong homology implies the vanishing of the derived limits of $\mathbf{A}$, and that in the presence of the continuum hypothesis, the first derived limit of $\mathbf{A}$ does not vanish. 
The following year, Dow, Simon, and Vaughan showed in \cite{DSV} that the first derived limit of $\mathbf{A}$ does vanish in the presence of the Proper Forcing Axiom, a set theoretic assumption widely believed to have the consistency strength of a supercompact cardinal. 
Todor\v{c}evi\'c  in \cite{PPT} reduced the hypothesis to the Open Graph Axiom, a consequence of the Proper Forcing Axiom which has no large cardinal strength. 

Beginning with Bergfalk's investigations in \cite{SHDLST}, the derived limits of $\mathbf{A}$ and its relatives have received renewed addition, and numerous results have shown that for every $1\leq n<\omega$, the (non)vanishing of $\lim^n\mathbf{A}$ has considerable set-theoretic content. Additionally, investigations of the original motivations from strong homology have led to consistent additivity results in the same models where the derived limits of $\mathbf{A}$ vanish. 
A (very noncomprehensive) collection of results includes:
\begin{itemize}
    \item In \cite{SHDLST}, Bergfalk shows that the vanishing of $\lim^2\mathbf{A}$ is independent of the axioms of set theory.
    \item In \cite{SVHDL}, Bergfalk and Lambie-Hanson show that it is consistent relative to a weakly compact cardinal, $\lim^n\mathbf{A}=0$ for every $1\leq n<\omega$.
    \item In \cite{ASH}, Bannister, Bergfalk, and Moore show that in the model produced by Bergfalk and Lambie-Hanson, strong homology is additive and has compact supports on the class of locally compact separable metric spaces. 
    \item In \cite{SVHDLWLC}, Bergfalk, Hru\v{s}\'ak, and Lambie-Hanson remove the large cardinal hypothesis of \cite{SVHDL} to obtain a model where $\lim^n\mathbf{A}=0$ for every $1\leq n<\omega$. 
    \item In \cite{ADLCM}, Bannister shows that in the model produced by Bergfalk, Hru\v{s}\'ak, and Lambie-Hanson, strong homology is additive and has compact supports on the class of locally compact separable metric spaces. 
    \item In \cite{NVHDL}, Veli\v{c}kovi\v{c} and Vignati show that for every $1\leq n<\omega$, the vanishing of $\lim^n\mathbf{A}$ is independent of the axioms of set theory. 
\end{itemize}
Note there is a trend in these results: first, a vanishing result about the derived limits of $\mathbf{A}$ then a result about strong homology being additive and having compact supports in the same model. 
In this paper, we show that this is no accident: the vanishing of derived limits of $\mathbf{A}$ \emph{implies} that strong homology is additive and has compact supports on the class of locally compact separable metric spaces. 
Moreover, a similar result holds for the ``wider'' system $\mathbf{A}_\kappa$.
To be precise, we show:
\begin{thm} \label{main_intro}
    For every cardinal $\kappa$, strong homology is additive and has compact supports on the class of locally compact metric spaces of weight at most $\kappa$ if and only if $\lim^n\mathbf{A}_\kappa=0$ for every $1\leq n<\omega$. 
\end{thm}
Theorem \ref{main_intro} completes a long circle of implications from \cite{ASH}, \cite{ADLCM}, rendering them equivalences, which we summarize as Theorem \ref{circle}. Already appearing in \cite{ASH} is a proof that items (\ref{cs_item}) through (\ref{add_lim_item}) are equivalent in the special case $\kappa=\omega$ and then generalized to arbitrary $\kappa$ in \cite{ADLCM}; see Theorem \ref{ADLCM_thm} below. 
Bergfalk isolated the notion of an $n$-coherent family of functions in \cite{SHDLST} and showed that (\ref{limA_item}) and (\ref{coh_item}) are equivalent. 
The implications (\ref{add_lim_item}) implies (\ref{om_coh_item}) and (\ref{om_coh_item}) implies (\ref{coh_item}) are immediate. 
We will show in Section \ref{proof_section} that (\ref{limA_item}) implies (\ref{add_lim_item}), thereby completing the proof of Theorem \ref{circle}. 
We define $\Om_\kappa$ systems below; see \cite{ASH} for a definition of their corresponding coherent families and the definition of type $II$ triviality. 
We note that the implication (\ref{om_coh_item}) implies (\ref{additivity_item}) answers a question posed in \cite[Remark 2]{ASH}.

\begin{thm} \label{circle}
    The following are equivalent:
    \begin{enumerate}
        \item $\lim^n\mathbf{A}_\kappa=0$ for every $1\leq n<\omega$. \label{limA_item}
        \item For every $1\leq n<\omega$, every $n$-coherent family of functions indexed by $\omega^\kappa$ is trivial. \label{coh_item}
        \item Whenever $X$ is a locally compact metric space of weight at most $\kappa$, 
\[\overline{H}_n(X)\cong\underset{\substack{K\subseteq X\\K\text{ compact}}}{\operatorname{colim}}\overline{H}_n(K),\]
where $\overline{H}_n$ is strong homology (recall that the weight of a topological space is the minimum cardinality of a basis).
That is, strong homology has \emph{compact supports} on the class of locally compact separable metric spaces. 
\label{cs_item}
        \item Whenever $\langle X_i\mid i<\kappa\rangle$ are locally compact metric spaces of weight at most $\kappa$, the natural map
\[\bigoplus_{i<\kappa}\overline{H}_n(X_i)\to\overline{H}_n\left(\coprod_{i<\kappa }X_i\right)\]
is an isomorphism; that is, strong homology is \emph{additive} on the class of locally compact metric spaces of weight at most $\kappa$. 
\label{additivity_item}
        
        \item Whenever $\langle X_i\mid i<\kappa\rangle$ are compact metric and $p\geq0$, the canonical map
        \[\bigoplus_{i<\kappa}\overline H_n(X_i)\to\overline{H}_n\left(\coprod_{i<\kappa}X_i\right)\]
        is an isomorphism. 
        \item Whenever $\mathcal{G}$ is an $\mathit{\Om_\kappa}$ system with each group $G_{\alpha,k}$ finitely generated and $n\geq0$, the canonical map
        \[\bigoplus_{\alpha<\kappa}\lim\hspace{-.1em}^n\,\G_\alpha\to\lim\hspace{-.1em}^n\, \G\]
        is an isomorphism. \label{add_lim_item}
        \item Whenever $\mathcal{G}$ is an $\mathit{\Om_\kappa}$ system with each $G_{\alpha,k}$ finitely generated and $n\geq1$, every $n$-coherent family corresponding to $\mathcal{G}$ is type $II$ trivial. \label{om_coh_item}
        
    \end{enumerate}
\end{thm}
\section{Preliminaries}

\begin{defn} \label{Om_kappa sys}
Suppose $\kappa$ is a cardinal. 
An \emph{$\Omega_{\kappa}$ system} $\mathcal{G}$ is specified by an indexed collection $\{G_{\alpha,k} \mid \alpha<\kappa, k \in \omega\}$
of abelian groups along with, for $\alpha<\kappa$ and $j\geq k$, compatible homomorphisms
$p_{\alpha,j,k}: G_{\alpha,j} \to G_{\alpha,k}$.

Such data give rise to the following additional objects:
\begin{itemize}

\item 
For each $x \in \omega^\kappa$ define 
${\displaystyle G_x := \bigoplus_{\alpha<\kappa} G_{\alpha,x(\alpha)}}$

\item
For each $x \leq y \in \omega^\kappa$ a homomorphism
$p_{y,x}:G_y \to G_x$
defined by $p_{y,x} := \bigoplus_{\alpha<\kappa} p_{\alpha,y(\alpha),x(\alpha)}$.

\item 
The systems $\mathbf{G}$ indexed over $\omega^\kappa$ with structure given by the above points.

\item 
For each $\alpha<\kappa$ an inverse system $\G_\alpha$ indexed over $\omega$ with $(\G_\alpha)_k=G_{\alpha,k}$ and structure maps given by $p_{\alpha,j,k}$. 
We will often abbreviate $p_{\alpha,k+1,k}$ as $p_{\alpha,k}$. 
We denote the canonical map from $\lim \G_\alpha$ to $G_{\alpha,k}$ as $p_{\alpha,\omega,k}$.

\end{itemize}
\end{defn}
For each $\alpha<\kappa$, the map from $\omega^\kappa$ to $\omega$ given by evaluation at $\alpha$ induces a functor from inverse systems indexed by $\omega$ to those indexed by $\omega^\kappa$. 
This functor commutes with $\lim$ and preserves both exact sequences and injective objects and therefore preserves derived limits; see \cite[Theorem 14.9]{SSH}. 
The canonical inclusion from the pulled back version of $\G_\alpha$ to $\G$ induces a map of derived limits from $\lim^n\G_\alpha$ to $\G$ so that there is a canonical map from $\bigoplus_{\alpha<\kappa}\lim\hspace{-.1em}^n\,\G_\alpha$ to $\lim\hspace{-.1em}^n\,\mathbf{G}$. 
A question bearing directly on strong homology computations is whether this map is always an isomorphism, as the following theorem indicates:
\begin{thm}[B. {\cite[Theorem 1.3]{ADLCM}}] \label{ADLCM_thm}
    The following are equivalent:
    \begin{enumerate}
\item Strong homology has compact supports on the class of locally compact metric spaces of weight at most $\kappa$. \label{cs}
    \item Strong homology is additive on the class of locally compact metric spaces of weight at most $\kappa$. 
\item Whenever $\langle X_i\mid i<\kappa\rangle$ are compact metric spaces, the natural map 
\[\bigoplus_{i<\kappa}\overline{H}_n(X_i)\to\overline{H}_n\left(\coprod_{i<\kappa }X_i\right)\]
is an isomorphism. \label{add_c}
        \item Whenever $\mathcal{G}$ is an $\Om_{\kappa}$ system with all groups finitely generated, the canonical map 
\[\bigoplus_{\alpha<\kappa}\lim\hspace{-.1em}^n\,\G_\alpha\to \lim\hspace{-.1em}^n\,G\]
is an isomorphism. \label{add_om}
    \end{enumerate}
\end{thm}
One $\Om_\kappa$ system of particular importance is the system $\mathcal{A}_\kappa$ where $(\mathcal{A}_\kappa)_{\alpha,k}=\mathbb{Z}^k$ with the canonical projection maps. 
We note that for every $\alpha<\kappa$ and $1\leq n<\omega$,  $\lim\hspace{-.1em}^n\,(\mathbf{A}_\kappa)_\alpha=0$ so that the additivity of derived limits for $\mathcal{A}_\kappa$ is equivalent to the vanishing of derived limits of $\mathbf{A}_\kappa$. 
In turn, the vanishing of derived limits of $\mathbf{A}_\kappa$ has a nice set-theoretic characterization in terms of coherent families of functions being trivial, though we will not need this characterization. 
See \cite[Theorem 3.3]{SHDLST} for a statement and proof. 

We will see that the vanishing of all higher derived limits of the corresponding system $\mathbf{A}_\kappa$ holds implications for the additivity of derived limits for all $\Om_\kappa$ systems and therefore by Theorem \ref{ADLCM_thm} for the additivity of strong homology. 
\section{The proof} \label{proof_section}
This section consists of a proof of the following theorem to complete the circle of implications:
\begin{thm} \label{main}
    Suppose that $\lim^s\mathbf{A}_\kappa=0$ for all $1\leq s\leq n+1$. 
    Then whenever $\mathcal{G}$ is an $\Om_\kappa$ system with each $G_{\alpha,k}$ finitely generated, the canonical map 
    \[\bigoplus_{\alpha<\kappa}\lim\hspace{-.1em}^n\,\G_\alpha\to\lim\hspace{-.1em}^n\,\mathbf{G}\]
    is an isomorphism. 
\end{thm}
Our first reduction is from general $\Om_\kappa$ systems to a more restricted class of \emph{$\mathbf{A}_\kappa$-like} systems. 
\begin{defn}
    An $\Om_\kappa$ system $\mathcal{G}$ is \emph{$\mathbf{A}_\kappa$-like} if there are finitely generated abelian groups $\langle H_{\alpha,k}\mid \alpha<\kappa,k<\omega\rangle$ such that for all $\alpha,k$, $G_{\alpha,k}\cong\prod_{i\leq k}H_{\alpha,k}$ with the maps appearing in $\mathcal{G}$ the canonical projection maps. 

    If $\mathcal{G}$ is any $\Om_\kappa$ system with each $G_{\alpha,k}$ finitely generated, we define the associated $\mathbf{A}_\kappa$-like system $A^\mathcal{G}$ by setting $H_{\alpha,k}=G_{\alpha,k}$. 
    Note that there is a canonical inclusion map $i^\mathcal{G}\colon\mathcal{G}\to A^\mathcal{G}$ given by $i^\mathcal{G}_{\alpha,k}=\prod_{i\leq k}p^\mathcal{G}_{\alpha,k,i}$. 
\end{defn}

Note that $\mathcal{A}_\kappa$ is $\mathbf{A}_\kappa$-like with $H_{\alpha,k}=\mathbb{Z}$ for each $\alpha<\kappa,k<\omega$. We now reduce to $\mathbf{A}_\kappa$-like systems. 
\begin{lem} \label{Ak_like_lem}
    Suppose $n<\omega$ and that $\lim^s\mathbf{H}=0$ whenever $\mathcal{H}$ is $\mathbf{A}$-like and $1\leq s\leq n$. 
    Then whenever $\mathcal{G}$ is an $\Om_\kappa$ system with each $G_{\alpha,k}$ finitely generated, the canonical map \[\bigoplus_{\alpha<\kappa}\lim\hspace{-.1em}^n\,\G_\alpha\to\lim\hspace{-.1em}^n\,\mathbf{G}\]
    is an isomorphism. 
\end{lem}
\begin{proof}
    By induction on $n$. 
    When $n=0$, the conclusion is a ZFC fact (see \cite[Theorem 9]{SHINA}). 
    Now let $1\leq n<\omega$ and fix an $\Om_\kappa$ system $\mathcal{G}$. 
    The short exact sequence of systems
    \[0\to\G\to\mathbf{A}^\mathcal{G}\to\mathbf{A}^\mathcal{G}/\mathcal{G}\to0\]
    as well as similar sequences at each $\alpha$ induces a diagram with exact rows of the form
    \begin{center}
\begin{tikzcd}
\bigoplus_{\alpha<\kappa}\lim^{n-1}{\mathbf{A}^{\mathcal{G}}_\alpha} \arrow[r] \arrow[d] & \bigoplus_{\alpha<\kappa}\lim^{n-1}({\mathbf{A}^{\mathcal{G}}_\alpha/\G_\alpha)} \arrow[r] \arrow[d] & \bigoplus_{\alpha<\kappa}\lim^n\G_\alpha \arrow[r] \arrow[d] & \bigoplus_{\alpha<\kappa}\lim^n\mathbf{A}^{\mathcal{G}}_i \arrow[d] \\
\lim^{n-1}\mathbf{A}^{\mathcal{G}} \arrow[r]                                   & \lim^{n-1}(\mathbf{A}^{\mathcal{G}}/\mathbf{G}) \arrow[r]                                     & \lim^n\mathbf{G} \arrow[r]                                 & \lim^n\mathbf{A}^{\mathcal{G}}                                
\end{tikzcd}
    \end{center}
    
By the inductive hypothesis, the first two vertical maps are isomorphisms and by hypothesis the two rightmost groups are $0$. 
Therefore the desired map is an isomorphism by the five lemma. 
\end{proof}
Our next reduction is from $\mathbf{A}_\kappa$-like systems to systems which are \emph{essentially $\mathbf{A}_\kappa$}. 
\begin{defn}
    An $\mathbf{A}_\kappa$-like system $\mathcal{G}$ is \emph{essentially $\mathbf{A}_\kappa$} if additionally each $H_{\alpha,k}$ is free and nonzero. 
    That is, there are nonzero finitely generated free abelian groups $H_{\alpha,k}$ such that $G_{\alpha,k}\cong\prod_{i\leq k}H_{\alpha,k}$ with maps corresponding to the projection maps. 
\end{defn}
The rationale behind the name choice is the following:
\begin{prop} \label{ess_A_prop}
    Suppose $\mathcal{G}$ is essentially $\mathbf{A}_\kappa$. 
    There is a cofinal $X\subseteq\omega^\kappa$ with an isomorphism of posets $\varphi\colon X\cong\omega^\kappa$ and compatible isomorphisms of abelian groups $(\mathbf{A}_\kappa)_x\cong G_{\varphi(x)}$ for each $x\in X$. 
\end{prop}
\begin{proof}
    Let 
    \[X=\{x\in\omega^\kappa\mid \forall \alpha<\kappa\;\exists k<\omega\;(x(\alpha)=\operatorname{rk}(G_{\alpha,k}))\},\]
    where $\operatorname{rk}(G_{\alpha,k})$ is the unique $\ell$ such that $\operatorname{rk}(G_{\alpha,k})\cong\mathbb{Z}^\ell$. 
    Note that since the $H_{\alpha,k}$ are nonzero for any system which is essentially $\mathbf{A}_\kappa$, for each $x\in X$ and $\alpha<\kappa$ there is exactly one such $k$. 
    In particular, the function $\varphi$ defined by $\varphi(x)= (\alpha\mapsto\operatorname{rk}(G_{\alpha,x(\alpha)}))$ defines an order-preserving bijection between $\omega^\kappa$ and $X$. 
    Moreover, for each $\alpha,k$ and $x\in\omega^\omega$, we may readily define compatible isomorphism from $\mathbf{G}_x$ to $(\mathbf{A}_\kappa)_{\varphi(x)}$ on the generators. 
\end{proof}
We now make use of the standard fact that derived limits may be computed along any cofinal suborder (see \cite[Theorem 14.9]{SSH}) to conclude the following.

\begin{cor} \label{ess_A_cor}
    Whenever $\mathcal{G}$ is essentially $\mathbf{A}_\kappa$ and $n<\omega$, $\lim^n\mathbf{G}\cong\lim^n\mathbf{A}_\kappa$.
\end{cor}
In light of Corollary \ref{ess_A_cor} and Lemma \ref{Ak_like_lem}, to complete the proof of Theorem \ref{main}, we need only prove the following lemma: 
\begin{lem} \label{A_like_ess_A_lem}
    Suppose $1\leq n<\omega$ and whenever $\mathcal{G}$ is essentially $\mathbf{A}_\kappa$, $\lim^n\mathbf{G}=\lim^{n+1}\mathbf{G}=0$. 
    Then whenever $\mathcal{H}$ is $\mathbf{A}_\kappa$-like, $\lim^n\mathbf{H}=0$.
\end{lem}
\begin{proof}
    
    The key claim is the following; note that a map of inverse systems if epic if and only if every component is a surjection:
    \begin{claim}
        Suppose that $\mathcal{G}$ is an $\mathbf{A}_\kappa$-like system. 
    There is an essentially $\mathbf{A}_\kappa$ system $\mathcal{F}$ and an epic $\varphi\colon\mathcal{F}\to\mathcal{G}$ such that $\ker(\varphi)$ is also essentially $\mathbf{A}_\kappa$.  
    \end{claim}
    \begin{proof}
        
    Let $\mathcal{G}$ be induced by the groups $\langle H_{\alpha,k}\mid \alpha<\kappa,k<\omega\rangle$. 
    For each $\alpha,k$, let $F_{\alpha,k},\psi_{n,k}$ be such that
    \begin{itemize}
        \item $F_{\alpha,k}$ is a finitely generated nonzero free abelian group.
        \item $\psi_{\alpha,k}\colon F_{\alpha,k}\to H_{\alpha,k}$ is surjective with a nonzero kernel.
    \end{itemize}
    Then let $\mathcal{F}$ be induced by the $F_{n,k}$ and let $\varphi_{\alpha,k}\colon\prod_{i\leq k}F_{\alpha,k}\to \prod_{i\leq k}H_{\alpha,k}$ be $\prod_{i\leq k}\psi_{\alpha,i}$.
    Then $\mathcal{F}$ is an essentially $\mathbf{A}_\kappa$ system. 
    Moreover, $\ker(\varphi)$ is the $\mathbf{A}_\kappa$-like system induced by $\ker(\psi_{n,k})$ and therefore essentially $\mathbf{A}_\kappa$ since a subgroup of a free group is free.
    \end{proof}
    With the claim in hand, the proof of Lemma \ref{A_like_ess_A_lem} follows quickly. 
    Given $\mathcal{G}$ which is $\mathbf{A}_\kappa$-like, fix essentially $\mathbf{A}_\kappa$ systems $\mathcal{F},\mathcal{F}'$ and a short exact sequence
    \begin{center}
        \begin{tikzcd}
            0\arrow[r]&\mathcal{F}\arrow[r]&\mathcal{F}'\arrow[r]&\mathcal{G}\arrow[r]&0.
        \end{tikzcd}
    \end{center}
    The corresponding long exact sequence of derived limits yields a sequence
    \begin{center}
        \begin{tikzcd}
            \lim^n\mathbf{F}'\arrow[r]&\lim^{n}\G\arrow[r]&\lim^{n+1}\mathbf{F}.
        \end{tikzcd}
    \end{center}
    By hypothesis, the first and last groups are $0$ so $\lim^n\mathbf{G}=0$ by exactness.
\end{proof}
\section{Questions}
We now conclude with some questions that remain open. 
We first ask whether the hypotheses can all be obtained simultaneously:
\begin{question} \label{main_quest}
    Is it consistent that for every cardinal $\kappa$ and every $1\leq n<\omega$, $\lim^n\mathbf{A}_\kappa=0$?
    Equivalently, is it consistent that strong homology is additive and has compact supports on the class of locally compact metric spaces?
\end{question}
We recall that by \cite[Theorem 1.2]{ADLCM} that for any cardinal $\kappa$, there is a forcing extension in which $\lim^n\mathbf{A}_\kappa=0$ for all $1\leq n<\omega$, but this forcing adds many reals. 
A specific instance of Question \ref{main_quest} of interest is the following:
\begin{question}
    Is it consistent that $\lim^2\mathbf{A}_{2^{\aleph_0}}=0$?
\end{question}
We note here that the proof of Theorem \ref{main} generalizes to show that the vanishing of $\lim^n\mathbf{A}_\kappa[\bigoplus_{i<\lambda}\mathbb{Z}]$ for every cardinal $\lambda$ and $n<\omega$ implies the additivity of derived limits for \emph{all} $\Om_\kappa$ systems. 
In this light, a strengthening of Question \ref{main_quest} is the following:
\begin{question}
    Is it consistent that for every cardinal $\kappa$, derived limits are additive for all $\Om_\kappa$ systems?
\end{question}
In both models where we know the derived limits of $\mathbf{A}$ simultaneously vanish, the same holds for the systems $\mathbf{A}[H]$ for any abelian group $H$ (see \cite[Theorem 7.7]{DAHDL} and \cite[Theorem 1.2]{ADLCM}). 
The following seems natural to ask:
\begin{question} \label{A0_implies_AH0_quest}
    Does $\lim^n\mathbf{A}=0$ for all $1\leq n<\omega$ imply that $\lim^n\mathbf{A}[H]=0$ for all $1\leq n<\omega$ and all abelian groups $H$?
\end{question}
One major open question in the theory of $\lim^n\mathbf{A}$ is the following. 
The smallest known upper bound is $\aleph_{\omega+1}$, obtained by Bergfalk, Hru\v{s}\'ak, and Lambie-Hanson in \cite{SVHDLWLC}. 
In light of \cite[Theorem A(1)]{SNVHDL}, a positive answer to Question \ref{A0_implies_AH0_quest} would yield that $\aleph_{\omega+1}$ is optimal. 
\begin{question}
    What is the least value of $2^{\aleph_0}$ compatible with the assertion that $\lim^n\mathbf{A}=0$ for every $1\leq n<\omega$?
\end{question}

\end{document}